\documentclass[12pt,reqno]{amsart}

\usepackage{amssymb}
\usepackage{graphicx}

\numberwithin{equation}{section}

\def\ZZ{{\mathbb Z}}

\def\Stdepth{\operatorname{Stdepth}}
\def\Hdepth{\operatorname{Hdepth}}
\def\depth{\operatorname{depth}}
\def\id{\operatorname{id}}
\def\rank{\operatorname{rank}}
\def\Ann{\operatorname{Ann}}
\def\supp{\operatorname{supp}}

\let\Dirsum=\bigoplus
\let\dirsum=\oplus
\let\tensor=\otimes
\let\iso=\cong

\def\cD{{\mathcal D}}
\def\cH{{\mathcal H}}

\def\cG{{\mathcal G}}
\def\cK{{\mathcal K}}

\def\mm{{\mathfrak m}}
\def\nn{{\mathfrak n}}
\def\pp{{\mathfrak p}}

\def\sqf{\operatorname{sqf}}

\newtheorem{lemma}{Lemma}[section]
\newtheorem{corollary}[lemma]{Corollary}
\newtheorem{theorem}[lemma]{Theorem}
\newtheorem{proposition}[lemma]{Proposition}

\theoremstyle{definition}
\newtheorem{definition}[lemma]{Definition}
\newtheorem{remark}[lemma]{Remark}
\newtheorem{remarks}[lemma]{Remarks}

\title{Stanley decompositions and Hilbert depth in the Koszul
complex}

\author{Winfried Bruns}
\address{Universit\"at Osnabr\"uck, Institut f\"ur Mathematik, 49069 Osnabr\"uck, Germany}
\email{wbruns@uos.de}

\author{Christian Krattenthaler$^{\dagger}$}
\address{Fakult\"at f\"ur Mathematik, Universit\"at Wien,
Nordbergstra{\ss}e~15, A-1090 Vienna, Austria.
WWW: \tt http://www.mat.univie.ac.at/\lower0.5ex\hbox{\~{}}kratt.}

\author{Jan Uliczka}
\address{Universit\"at Osnabr\"uck, Institut f\"ur Mathematik, 49069 Osnabr\"uck, Germany}
\email{Jan.Uliczka@uos.de}

\dedicatory{To Ralf Fr\"oberg on his 65th birthday}

\begin{document}

\thanks{$^\dagger$Research partially supported by the Austrian
Science Foundation FWF, grants Z130-N13 and S9607-N13,
the latter in the framework of the National Research Network
``Analytic Combinatorics and Probabilistic Number Theory"}

\begin{abstract}
Stanley decompositions of multigraded modules $M$ over polynomials rings
have been discussed intensively in recent years. There is a natural
notion of depth that goes with a Stanley decomposition, called the
{\it Stanley depth}. Stanley conjectured that the Stanley depth of a module
$M$ is always at least the (classical) depth of $M$. In this paper we
introduce a weaker type of decomposition, which we call {\it Hilbert
decomposition}, since it only depends on the Hilbert function of $M$, and
an analogous notion of depth, called {\it Hilbert depth}. Since Stanley
decompositions are Hilbert decompositions, the latter set upper bounds
to the existence of Stanley decompositions. The advantage of Hilbert
decompositions is that
they are easier to find. We test our new notion on the syzygy modules
of the residue class field of $K[X_1,\dots,X_n]$ (as usual identified with
$K$). Writing $M(n,k)$ for the $k$-th syzygy module, we show that the
Hilbert depth of $M(n,1)$ is $\lfloor(n+1)/2\rfloor$.
Furthermore, we show that, for $n > k \ge \lfloor n/2\rfloor$,
the Hilbert depth of $M(n,k)$ is equal to $n-1$. We
conjecture that the same holds for the Stanley depth. For the range
$n/2 > k > 1$, it seems impossible to come up with a compact formula for
the Hilbert depth. Instead, we provide very precise asymptotic results
as $n$ becomes large.
\end{abstract}

 \maketitle

 \section{Introduction}

In recent years \emph{Stanley decompositions} of multigraded modules
over polynomial rings $R=K[X_1,\dots,X_n]$ have been discussed
intensively. Such decompositions, introduced by Stanley in \cite{S},
break the module $M$ into a direct sum of graded vector subspaces,
each of which is of type $Sx$ where $x$ is a homogeneous element and
$S=K[X_{i_1},\dots,X_{i_d}]$ is a polynomial subalgebra. Stanley
conjectured that one can always find such a decomposition in which
$d\ge \depth M$ for each summand. (For unexplained terminology of
commutative algebra we refer the reader to \cite{BH}.)

One says that $M$ has \emph{Stanley depth} $m$, $\Stdepth M=m$,  if
one can find a Stanley decomposition in which $d\ge m$ for each
polynomial subalgebra involved, but none with $m$ replaced by $m+1$.
With this notation, Stanley's conjecture says $\Stdepth M\ge \depth
M$.

In this paper we introduce a weaker type of decomposition in which
we no longer require the summands to be submodules of $M$, but only
vector spaces isomorphic to polynomial subrings. Evidently, such
decompositions depend only on the Hilbert series of $M$, and
therefore they are called \emph{Hilbert decompositions}. The
\emph{Hilbert depth} $\Hdepth M$ is defined accordingly.

Since Stanley decompositions are Hilbert decompositions, the latter
set upper bounds to the existence of Stanley decompositions, and
since they are easier to find, one may try to construct a Stanley
decomposition by appropriately modifying a ``good'' Hilbert
decomposition.

Moreover, our discussion shows that it is worthwhile to also
consider the standard grading along with the multigrading, as
already suggested implicitly by Stanley, who allows arbitrary
gradings in his conjecture. In order to distinguish multigraded and
standard graded invariants, we use the indices $n$ and $1$,
respectively. All this is made precise in Section~\ref{StanHilb}. In
addition, in the same section, we collect several useful results
from the literature that are proved in a concise way, some of which
in an extended form.

While most papers are devoted to the case in which multigraded
components have $K$-dimension $\le 1$ (and in which Hilbert
decompositions and Stanley decompositions coincide under a mild
hypothesis), we test our notions on the syzygy modules of the
residue class field of $K[X_1,\dots,X_n]$ (as usual identified with
$K$). The Stanley depth of the first syzygy module, the maximal
ideal $\mm=(X_1,\dots,X_n)$, was found by Bir\'{o} et al.: $\Stdepth
\mm=\lfloor (n+1)/2\rfloor$. By the standard inductive approach to
the Koszul complex, it is then easily shown that the $k$-th syzygy
module $M(n,k)$ has multigraded Stanley depth $\ge \lfloor
(n+k)/2\rfloor$.

The further investigations reveal a significant difference between
the ``lower'' syzygy modules $M(n,k)$, $1<k<\lfloor n/2\rfloor$, and
the ``upper'' ones. For the upper ones, one can easily determine the
multigraded Hilbert depth: if $\lfloor n/2\rfloor<k<n$, then
$\Hdepth_n M(n,k)=n-1$, which is the best possible value for a
nonfree module. We believe that the multigraded Stanley depth has
the same value, and show that this holds for $k=n-3$ (in addition to
$k=n-2,n-1$).

In the lower range, it seems impossible to find a simple expression
even for $\Hdepth_1$, since the binomial sum that must be evaluated
for a precise bound (see Proposition~\ref{summation} and
Remark~\ref{rem:imposs}) cannot be summed in closed form. The best
we can offer in Section~\ref{Asympt} (apart from experimental values
for $n\le 22$) is asymptotic estimates for $\Hdepth_1 M(n,k)$ as $n$
becomes large. We consider two ``regimes": if $k$ is fixed and $n$
tends to $\infty$, then Theorem~\ref{thm:regimeA} provides a rather
precise asymptotic approximation, showing in particular that the
lower bound $\lfloor (n+k)/2\rfloor\sim n/2$ has the correct leading
asymptotic order, although it is still rather far away from the true
value. This changes, if both $k$ and $n$ tend to $\infty$ at a fixed
rate: as we show in Theorem~\ref{thm:regimeB}, in that case
$\Hdepth_1 M(n,k)\sim \varepsilon n$ with $\varepsilon>1/2$, where
$\varepsilon$ depends on the ratio between $k$ and $n$. In
particular, again, this turns out to be much larger than the
corresponding value provided by the lower bound $\lfloor
(n+k)/2\rfloor$ (see Remarks~\ref{rem:regimeB}.(2)).

\section{Stanley decompositions and Hilbert depth}\label{StanHilb}

We consider the polynomial ring $R=K[X_1,\dots,X_n]$ over a field
$K$ and two graded structures on $R$:
\begin{enumerate}
\item the \emph{multigrading}, more precisely, the $\ZZ^n$-grading
in which the degree of $X_i$ is the $i$-th vector $e_i$ of the
canonical basis;
\item the \emph{standard grading} over $\ZZ$ in which each $X_i$ has
degree $1$.
\end{enumerate}
All $R$-modules are assumed to be finitely generated.

In order to treat both cases in a uniform way, we use \emph{graded
retracts} of $R$, namely subalgebras $S\subset R$ such that there
exists a graded epimorphism $\pi:R\to S$
with $\pi|S=\id$. In the multigraded case, these retracts are the
subalgebras generated by a subset of the indeterminates, and, in the
standard graded case, they are the subalgebras generated by a set of
$1$-forms.

\begin{definition}\label{Stdef}
Let $M$ be a finitely generated graded $R$-module. A \emph{Stanley
decomposition} of $M$ is a finite family
$$
\cD=(S_i,x_i)_{i\in I}
$$
in which $x_i$ is a homogeneous element of $M$ and $S_i$ is a graded
$K$-algebra retract of $R$ for each $i\in I$ such that $S_i\cap \Ann
x_i=0$, and
$$
M=\Dirsum_{i\in I} S_ix_i
$$
as a graded $K$-vector space. (For convenience, we set $\Stdepth
0=\infty$.)
\end{definition}

While $M$ is not decomposed as an $R$-module in the definition, the
direct sum itself carries the structure of an $R$-module and has a
well-defined depth. Following Herzog et al.\ \cite{HSY} we make the
following definition.

\begin{definition}
The \emph{Stanley depth} $\Stdepth M$ of $M$ is the maximal depth of
a Stanley decomposition of $M$.
\end{definition}

In the following we will use the index $n$ in order to denote
invariants associated with the multigrading, and the index $1$ for
those associated with the standard grading. If no index appears in a
statement, then it applies to both cases.

\begin{remark}
Stanley \cite{S} introduced decompositions as in
Definition~\ref{Stdef} and conjectured that
\begin{equation}
\Stdepth M\ge \depth M\label{Stconj}
\end{equation}
for all modules $M$. However, one should note that the
decompositions considered by us are more special than Stanley's
since he allows arbitrary gradings on the polynomial ring.

The reason for our more restrictive definition is that we want the
denominators of the Hilbert series of the rings $S_i$ to divide the
denominator of the Hilbert series of $R$.
\end{remark}

It is not hard to see that Stanley's conjecture holds in the
standard graded case, at least for infinite fields. It was actually
proved by Baclawski and Garsia \cite{BG} before the conjecture was
made; see also Theorem~\ref{St_st}. For the multigraded case,
Stanley decompositions have recently been investigated in several
papers: Bir\'{o} et al.\ \cite{BHKTY}, Cimpoea{\c{s}} \cite{C},
Herzog et al.\ \cite{HSY}, Popescu \cite{P}, and Rauf \cite{R}.

From the combinatorial viewpoint, a module is often only an algebraic
substrate of its Hilbert function, and we may ask what
decompositions a given Hilbert function can afford.

\begin{definition}
Under the same assumptions on $R$ and $M$ as above, a \emph{Hilbert
decomposition} is a finite family
$$
\cH=(S_i,s_i)_{i\in I}
$$
such that
$s_i\in \ZZ^m$ (where $m=1$ or $m=n$, respectively, depending on whether
we are in the standard graded or in the multigraded case),
$S_i$ is a graded $K$-algebra retract of
$R$ for each $i\in I$, and
$$
M\iso\Dirsum_{i\in I} S_i(-s_i)
$$
as a graded $K$-vector space.
\end{definition}

A Stanley decomposition breaks $M$ into a direct sum of submodules
over suitable subalgebras, whereas for a Hilbert decomposition we
only require an isomorphism to the direct sum of modules over such
subalgebras. Clearly, Hilbert decompositions of $M$ depend only on
the Hilbert function of $M$. As for Stanley decompositions, we can
define $\depth \cH$.

\begin{definition}
The \emph{Hilbert depth} $\Hdepth M$ of $M$ is the maximal depth of
a Hilbert decomposition of $M$.
\end{definition}

Weakening Stanley's conjecture, one may ask whether
\begin{equation}\label{HD1}
\Hdepth M\ge \depth M,
\end{equation}
or, equivalently,
\begin{equation}\label{HD2}
\Hdepth M=\max\{\depth N: H(N,\_)=H(M,\_)\}.
\end{equation}
(Here $H(N,\_)$ denotes the Hilbert function of $N$, $H(N,g)=\dim_K
N_g$ for all $g\in \ZZ^m$.) It is clear that \eqref{HD2} implies
\eqref{HD1}, and the converse holds since $M$ and $N$ share all
Hilbert decompositions. Moreover, a positive answer to Stanley's
conjecture would evidently imply \eqref{HD1}.

Hilbert series, in the standard as well as in the multigraded case,
are rational functions of type
$$
H_M(T)=\frac{Q_M(T)}{(1-T)^n}\quad\text{and}\quad
H_M(T_1,\dots,T_n)=\frac{Q_M(T_1,\dots,T_n)}{(1-T_1)\cdots(1-T_n)},
$$
respectively, where $Q_M(T)\in \ZZ[T^{\pm1}]$ and
$Q_M(T_1,\dots,T_n)\in\ZZ[T_1^{\pm1},\dots,T_n^{\pm1}]$ are Laurent
polynomials. A Hilbert decomposition in the standard graded case
amounts to a representation of the numerator in the form
$$
Q_M(T)=\sum_j q_j(T)\,(1-T)^{t_j}
$$
where $q_j$ is a Laurent polynomial with \emph{positive}
coefficients. Then the depth of the decomposition is $n-\max_j t_j$.
In the multigraded case, it amounts to a representation
$$
Q_M(T_1,\dots,T_n)=\sum_j q_j(T_1,\dots,T_n)\prod_{i\in I_j}(1-T_i),
$$
where the $I_j$'s are subsets of $\{1,\dots,n\}$, and the
polynomials $q_j$ are nonzero and have nonnegative coefficients.
Here, the depth of the decomposition is $n-\max_j\vert I_j\vert$.

Consider the following example: $R=K[X,Y]$, $M=K\dirsum
YR/(X)\dirsum YR$. Then
$$
H_M(T)=H_R(T)=\frac1{(1-T)^2}
$$
and
$$
H_M(T_1,T_2)=\frac{1-T_1+T_2}{(1-T_1)(1-T_2)}=
\frac1{1-T_2}+\frac{T_2}{(1-T_1)(1-T_2)}.
$$
It follows immediately that $\Hdepth_1 M=2$ and $\Hdepth_2 M=1$,
whereas\break $\Stdepth_1 M=\Stdepth_2 M=0$.

The following example, taken from \cite{U}, shows that $\Stdepth_n
M< \Stdepth_1 M$ in general. Let $R=K[X,Y,Z]$ and $M=R/(XZ,YZ,Z^2)$.
Then $\Stdepth_3 M =0$ by Remark~\ref{0converse} below, since
$\depth M=0$ and $\dim_K M_a \leq 1$ for all $a \in \ZZ^3$. On the other
hand,
$$
M \iso k[X] \cdot 1 + k[X]\cdot Y + k[Y]\cdot(Y+Z) + k[X,Y]\cdot XY^2,
$$
is a $\ZZ$-Stanley decomposition. Hence $\Stdepth_1 M \geq 1 =
\Hdepth_1 M$. To sum up,  $\Stdepth_1 M = 1 > 0 = \Stdepth_3 M$.

A priori, it is not clear that Stanley or Hilbert decompositions
exist at all. In the multigraded case one can use a standard
filtration argument. Under much more general assumptions, $M$ has a
filtration
$$
0=M_0\subset M_1\subset \dots\subset M_q=M
$$
in which each quotient $M_{i+1}/M_i$ is isomorphic to a shifted copy
$R/\pp_i(-m_i)$ of a residue class ring modulo a graded prime ideal
$\pp_i$. In the multigraded case, this fact establishes the existence
of Stanley decompositions, since each of the prime ideals $\pp_i$ is
generated by a subset of $X_1,\dots,X_n$.

\begin{proposition}\label{exseq}
Let
$$
0\to U\to M\to N\to 0
$$
be an exact sequence of graded $R$-modules. If $U$ and $N$ have
Stanley decompositions, then so does $M$, and
$$
\Stdepth M\ge \min(\Stdepth U, \Stdepth N).
$$
The same statements apply to Hilbert decompositions and depth.
\end{proposition}

\begin{proof}
For Hilbert decompositions, the statement is completely trivial
since $M$ and $U\dirsum N$ have the same Hilbert function. For
Stanley decompositions, it is only necessary to lift the generators
in a Stanley decomposition of $N$ to homogeneous preimages in $M$.
\end{proof}

In the standard graded case, a filtration as above does not yield a
Stanley decomposition since the residue class rings $R/\pp_i$ fail
to be retracts in general. This failure is however compensated by
the existence of Noether normalizations in degree $1$,
provided $K$ is infinite. By the following theorem of Baclawski and
Garsia \cite{BG}, Stanley decompositions exist in the standard
graded case, at least under a mild restriction, and inequality
\eqref{Stconj} holds. For the convenience of the reader, we include
the short proof.

\begin{theorem}\label{St_st}
Let $K$ be an infinite field. Then, in the standard graded case,
every $R$-module $M\neq 0$ has a Stanley decomposition, and
$$
\Stdepth_1 M\ge \depth M.
$$
\end{theorem}

\begin{proof}
If $\dim M=0$, the assertion is trivial, since $M$ is a
finite-dimensional $K$-vector space and $K$ is a retract of $R$.

Now suppose that $\dim M>0$. Note that for every graded $R$-module
there exists a homogeneous system of parameters $y_1,\dots,y_d$,
$d=\dim M$, in degree $1$. The essential point is that
$y_1,\dots,y_d$ generate a retract $S$ of $R$. Since all graded
retracts of $S$ are graded retracts of $R$, and since $\depth_S
M=\depth_R M$, we can replace $R$ by $S$. In other words, we may
assume that $\dim M=n$.

If $\depth M=n$, then $M$ is a free $R$-module, and the claim is
again obvious. Suppose that $\depth M< n$. Since $\dim M=\dim R$,
$M$ contains a free graded $R$-submodule $F$ of rank equal to $\rank
M$. Since $\depth M/F=\depth M$, but $\dim M/F<\dim M$, we can apply
induction.
\end{proof}

In the standard graded case, Hilbert decompositions were considered
by Uliczka \cite{U}. Among other things, he proved that
\begin{equation}\label{HD3}
\Hdepth M=n-\min\{u:Q_M(T)/(1-T)^u\text{ is positive}\}.
\end{equation}
Here $Q_M(T)$ is the numerator polynomial of the Hilbert series, and
a rational function is called \emph{positive} if its Laurent
expansion at $0$ has only nonnegative coefficients.

Our next result shows that, in the case that is certainly the most
interesting one from the combinatorial viewpoint, a Hilbert
decomposition is automatically a Stanley decomposition.

\begin{proposition}\label{multfree}
Suppose that $\dim_K M_t\le1$and $R_sM_t\neq0$ whenever
$R_s,M_t,\break M_{s+t}\neq0$. Let $\cH=(S_i,s_i)_{i\in I}$ be a
Hilbert decomposition of $M$, and choose a homogeneous nonzero
element $x_i\in M$ of degree $s_i$ for each $i$. Then
$\cD=(S_i,x_i)_{i\in I}$ is a Stanley decomposition of $M$.
\end{proposition}

The proof is straightforward: the supporting degrees of the vector
spaces $S_ix_i$ do not overlap since $\dim_K M_t\le1$ for all $t$,
and all degrees are reached.

In the general case, the choice of the elements $x_i$ is of course
critical. The next proposition gives a necessary and sufficient
condition.

\begin{proposition}\label{Hilb->St}
Let $\cH=(S_i,s_i)_{i\in I}$ be a Hilbert decomposition of $M$, and
choose a homogeneous nonzero element $x_i\in M$ of degree $s_i$ for
each $i$.
\begin{enumerate}
\item
The following properties are equivalent:

\begin{enumerate}
\item $\cD=(S_i,x_i)_{i\in I}$ is a Stanley decomposition.
\item If\/ $\sum_{i\in I} a_ix_i=0$ with $a_i\in
S_i$, then $a_i=0$ for all $i$.
\end{enumerate}

\item In particular, $\cD$ is a Stanley decomposition if for every
degree $g$ and the family $\cG=\{i: (S_ix_i)_g\neq 0\}$ the elements
$x_i$, $i\in \cG$, are linearly independent.
\end{enumerate}
\end{proposition}

In fact, the type of restricted linear independence in (1)(b) is
equivalent to the fact that the subspaces $S_ix_i$ form a direct
sum. Then they must ``fill'' $M$ since the direct sum has the same
Hilbert function as $M$. That (1)(b) follows from (2), results
immediately from the fact that every linear dependence relation of
homogeneous elements decomposes into its homogeneous components.

For a special case, the following proposition can be found in
\cite{R}.

\begin{proposition}\label{product}
Let $R$ and $S$ be polynomial rings over $K$, and let $M$ and $N$ be
graded modules over $R$ and $S$, respectively. Then
$$
\Stdepth M\tensor_K N\ge \Stdepth M + \Stdepth N,
$$
and the analogous inequality holds for $\Hdepth$. {\em(}Here,
$M\tensor_K N$ is considered as a module over $R\tensor_K S${\em)}.
\end{proposition}

The proposition is obvious since the tensor product is distributive
with respect to direct sums. The following proposition was proved in
\cite[1.8]{R} for the multigraded case.

\begin{proposition}\label{regseq}
With the standard assumptions on $R$ and $M$, suppose that
$a_1,\dots,a_r$ is a homogeneous $M$-sequence such that
$K[a_1,\dots,a_r]$ is a graded retract of $R$. Then
$$
\Stdepth M\ge \Stdepth M/(a_1,\dots,a_r)+r,
$$
and the analogous inequality holds for $\Hdepth$.
\end{proposition}

\begin{proof}
Suppose $\cD'=(S_i',x'_i)$ is a Stanley decomposition of
$M'=M/(a_1,\dots,a_r)$. Then we lift the $x_i'$ to homogeneous
elements of the same degree in $M$ and claim that
$\cD=(S_i[a_1,\dots,a_r],x_i)$ is a Stanley decomposition of $M$.

By induction it is enough to treat the case $r=1$. Let $R'=R/(a_1)$.
First one should convince oneself that, in the multigraded case,
$a_1$ is an indeterminate that does not occur in any of the $S_i'$.
Since $S_i'$ is a retract of $R$, the same holds for $S_i'[a_1]$; we
may assume $a_1=X_n$ in this case. In the standard graded case, we
choose subspaces $V_i$ of $R_1$ such that $\dim_K V_i=\dim_K (S'_i)$
and $V_i$ is mapped onto $(S_i')_1$ by the epimorphism $R\to R'\to
S_i'$. Clearly, $a_i\notin V_i$, and so $R_i$ is again a retract.

Since $H_M(T)=H_{M'}(T)/(1-T)$ in the standard graded case and
$H_M(T_1,\dots,T_n)=H_{M'}(T_1,\dots,T_{n-1},T_n)/(1-T_n)$ in the
multigraded case, our desired Stanley decomposition is at least a
Hilbert decomposition. (This argument proves the assertion about
Hilbert depth.)

We use Proposition~\ref{Hilb->St} to prove that it is indeed a
Stanley decomposition. Consider a critical relation
$b_1x_{i_1}+\dots+b_rx_{i_r}=0$, and expand each $b_i$ as a
polynomial in $a_1$ with coefficients in $S_i'$. Reduction modulo
$a_1$ yields that the constant terms of the $b_i$ must be zero, and
we can factor $a_1$ from the remaining terms. But $a_1$ is not a
zero divisor, and it can be cancelled. This reduces the $a_1$-degree
of our coefficients by $1$, and we are done.
\end{proof}

Note that Proposition~\ref{regseq} implies the inequality in
Theorem~\ref{St_st}; more precisely, it reduces the proof of the
theorem to the case where $\depth M=0$, since one can find a
suitable $M$-sequence of $1$-forms.

\begin{corollary}\label{syzygy}
Let $M$ be the $j$-th graded syzygy of a graded $R$-module $N$. Then
$\Stdepth M\ge j$.
\end{corollary}

For the proof it is enough to note that every $R$-sequence of length
$j$ is an $M$-sequence (see, for example, Bruns and Vetter
\cite[(16.33)]{BV}).

We use Proposition~\ref{regseq} to prove that Stanley's conjecture
holds in the multigraded case if $\depth M=1$. This was already
stated by Cimpoea{\c{s}} \cite{C}; however, the proof in \cite{C} is
not correct.

\begin{proposition}\label{depth0}
Suppose that $\depth M\ge 1$. Then $\Stdepth_n M\ge 1$.
\end{proposition}

\begin{proof}
Set $U_{n+1}=M$, $U_0=0$, and define
$$
U_i=\{x\in U_{i+1}: X_{i}^jx=0\text{ for some }j>0\}, \quad
i=1,\dots,n.
$$
Then we have a filtration of multigraded modules
$$
0=U_0 \subset U_1\subset \dots\subset U_{n+1}=M,
$$
so that $\Stdepth_n M\ge \min_i \Stdepth_n U_{i+1}/U_i$. Moreover,
$U_1/U_0\iso U_1=H_\mm^0(M)$ where $\mm=(X_1,\dots,X_n)$ and $H_\mm$
denotes local cohomology.

By hypothesis, $H_\mm^0(M)=0$, and, by construction, $X_i$ is not a
zero divisor of $U_{i+1}/U_i$ for $i=1,\dots,n$. Therefore
$\Stdepth_n M\ge 1$.
\end{proof}

\begin{remark}\label{0converse}
The converse of Proposition~\ref{depth0} does not hold in general,
as documented by the example given in \cite[1.6]{C}.

However, it is easy to see that $\Stdepth M=0$ if $H_\mm^0(M)$
contains a full graded component $M_g$ of $M$. For, then we must
have $H_\mm^0(M)\cap S_ix_i\neq 0$ for some component of the Stanley
decomposition, so that $H_\nn^0(S_i)
\iso H_\nn^0(S_ix_i)\neq 0$ for the ideal $\nn$ generated by the
indeterminates of $S_i$. This forces $S_i=K$.

Evidently, the assumption on $H_\mm^0(M)$ is satisfied if all
homogeneous components have dimension $\le 1$ over $K$, and for this
case this remark appeared already in \cite{C}.
\end{remark}

\section{The Koszul complex}\label{Koszul}

In the following we want to investigate the syzygy modules of $K$,
viewed as an $R$-module by identification with $R/\mm$,
$\mm=(X_1,\dots,X_n)$. With this $R$-module structure, $K$ is
resolved by the Koszul complex
$$
\cK(X_1,\dots,X_n;R):0\to\bigwedge^n
R^n\xrightarrow{\partial}\bigwedge^{n-1}R^n\xrightarrow{\partial}\dots
\xrightarrow{\partial} R^n\xrightarrow{\partial} R\to 0
$$
where the basis vector $e_{i_1}\wedge\dots\wedge e_{i_k}$ of
$\bigwedge^k R^n$ , $i_1<\dots<i_k$, has degree $X_{i_1}\cdots
X_{i_k}$ (we identify monomials with their exponent vectors when we
speak of degrees). In the standard grading, the degree of
$e_{i_1}\wedge\dots\wedge e_{i_k}$ is simply $k$.

Let $M(n,k)$ be the $k$-th syzygy module of $K$. The Hilbert series
of this module can be immediately read off the free resolution; its
numerator  polynomial is
\begin{equation}\label{NumSt}
\binom{n}{k}T^k-\binom{n}{k+1}T^{k+1}+\dots+(-1)^{n-k}T^n
\end{equation}
in the standard graded case, and
\begin{equation}\label{NumMult}
Q(n,k)=\sigma_{n,k}-\sigma_{n,k+1}+\dots+(-1)^{n-k}\sigma_{n,n}
\end{equation}
in the multigraded case, where $\sigma_{n,j}$ denotes the $j$-th
elementary symmetric polynomial in the indeterminates
$T_1,\dots,T_n$. Just for the record, the multigraded Hilbert series
of $M(n,k)$ is given by
\begin{equation}
H_{M(n,k)}(T_1,\dots,T_n)=\sum_{a\in\ZZ_+^n}
\binom{|\supp(a)|-1}{k-1} T_1^{a_1}\cdots T_n^{a_n};
\end{equation}
here $\supp(a)$ denotes the set of indices $i$ with $a_i\neq 0$. The
standard graded Hilbert series is contained in
Proposition~\ref{summation} (with $s=n$).

For $k=1$, one has the following result.

\begin{theorem}\label{Biro}
We have
$$
\Hdepth_1 \mm=\Hdepth_n \mm = \Stdepth_n \mm=\lfloor (n+1)/2\rfloor.
$$
\end{theorem}

\begin{proof}
For the difficult result on Stanley depth, see Bir\'{o} et al.\
\cite{BHKTY}. In order to estimate $\Hdepth_1$, one considers the
numerator polynomial of the Hilbert series,
$$
nT- \binom{n}{2}T^2 \pm\cdots.
$$
It is clear that we have to multiply by at least the power
$1/(1-T)^v$, with
$$
v=\left\lceil\frac{\binom{n}{2}}{n}\right\rceil=\lceil
(n-1)/2\rceil,
$$
in order to get a positive rational function. Hence, by \eqref{HD3},
$\lfloor (n+1)/2\rfloor$ is an upper bound for $\Hdepth_1 \mm$, and
the theorem follows. (See also \cite{U} for a direct computation of
$\Hdepth_1\mm$.)
\end{proof}

The Koszul complex allows (at least) two well-known inductive
approaches.

\begin{lemma}\label{KoszInd}
For all $n$ and $k$ one has
$$
\Stdepth M(n,k) \ge \Stdepth M(n-1,k)
$$
and
$$
\Stdepth M(n,k) \ge 1+\min\bigl\{\Stdepth M(n-1,k),\Stdepth
M(n-1,k-1)\bigr\},
$$
and the analogous inequalities hold for $\Hdepth$.
\end{lemma}

\begin{proof}
Here and in the following we will write $[i_1,\dots,i_k]$ for
$\partial(e_{i_1}\wedge\dots\wedge e_{i_k})$. Consider the submodule
$L$ of $M(n,k)$ generated by the elements $[i_1,\dots,i_{k-1},n]$.
An inspection of $M(n,k+1)$ yields that $M(n,k)/L$ is annihilated by
$X_n$. Thus $\rank M(n,k)/L=0$, and $\rank L=\rank
M(n,k)=\binom{n-1}{k-1}$. Since $L$ is generated by exactly this
number of elements, it is a free submodule. Therefore $\Stdepth
M(n,k)\ge \Stdepth M(n,k)/L$.

Let $R'=K[X_1,\dots,X_{n-1}]$. The natural epimorphism $R\to R'$
that sends $X_i$ to itself for $i\neq n$ and $X_n$ to $0$, can be
lifted to a chain map of the Koszul complexes that sends $e_i$ to
``itself'' and $e_n$ to $0$. This map induces an epimorphism
$M(n,k)/L\to M(n-1,k)$, which is an isomorphism since the modules
have the same Hilbert function. This proves the first inequality.

For the second inequality we use the inductive construction of the
Koszul complex by iterated tensor products over $R$ (see
\cite[1.6.12]{BH}):
$$
\cK(X_1,\dots,X_n;R)=\cK(X_1,\dots,X_{n-1};R)\tensor_R \cK(X_n;R).
$$
It yields an exact sequence
$$
0\to N(n-1,k)\to M(n,k) \to N(n-1,k-1) \to 0,
$$
where $N(n-1,j)$ is the $j$-th syzygy module of
$R/(X_1,\dots,X_{n-1})$. On the other hand,
$N(n-1,j)=M(n-1,j)\tensor_K K[X_n]$, and the inequality follows
from Propositions~\ref{exseq} and \ref{product}.
\end{proof}

If we combine Theorem~\ref{Biro} inductively with the second
inequality, then we obtain a significant improvement of the bound
$\Stdepth M(n,k)\ge k$ that one gets for free from Corollary~\ref{syzygy}.

\begin{corollary}\label{StdepthMk}
Let $M(n,k)$ be the $k$-th syzygy module of $K$. Then
$$
\Stdepth_n M(n,k)\ge \lfloor (n+k)/2\rfloor.
$$
\end{corollary}

\begin{remark}\label{monregseq}
Theorem~\ref{Biro} has been generalized to ideals generated by
monomial regular sequences $y_1,\dots,y_m$ as follows:
$\Stdepth_n (y_1,\dots,y_m)=n-\lfloor m/2\rfloor$; see Shen
\cite[2.4]{Ch}. Since $R/I$ is resolved by the Koszul complex
$\cK(y_1,\dots,y_m;R)$, a similar induction as in the proof of
Corollary~\ref{StdepthMk} shows that the $k$-th syzygy module
of $R/(y_1,\dots,y_m)$ has multigraded Stanley depth $\ge
n-m+\lfloor (m+k)/2\rfloor$. In the induction, one must observe
that the indeterminate factors of the $y_i$ form pairwise
disjoint sets.
\end{remark}

The upper half of the resolution poses no problems for Hilbert
depth.

\begin{theorem}\label{Hupper}
Suppose $n>k\ge\lfloor n/2\rfloor$. Then
$$
\Hdepth_1 M(n,k)=\Hdepth_n M(n,k)=n-1.
$$
\end{theorem}

\begin{proof}
Note that the maximal value $n$ is excluded. It can only be attained
by a module of Krull dimension $n$ with a positive numerator
polynomial in its Hilbert series, standard graded or multigraded. It
is therefore enough to consider the multigraded case.

Now we look at the multigraded numerator polynomial, given by
equation~\eqref{NumMult}. Consider the set $Y_u$ of squarefree
monomials in $T_1,\dots,T_n$ of degree $u$, summing up to
$\sigma_{nu}$. For $u\ge\lfloor n/2\rfloor$ one has an injective map
$Y_u\to Y_{u-1}$ that assigns each monomial a divisor (cf.\
\cite[p.~35]{StWhAA}). It follows that we can write $Q(n,k)$ as a
sum of monomials and polynomials of type
$\mu(1-T_p)$ where $\mu$ is a monomial. Exactly those terms
$\mu(1-T_p)$ appear for which $\mu$ is the image of $\mu T_p$ under
the injection.

This leads to a Hilbert decomposition in which the summands are of
type $R$ and $R/(X_p)$ (with appropriate shifts). More precisely,
the decomposition is given by
$$
(K[F'_i],X^{F_i}),
$$
where
\begin{itemize}
\item $F_i$ runs
through the subsets of $\{1,\dots,n\}$ with $k+j$
elements, $j$ even,
\item $X^{F_i}$ is the product of the indeterminates
    dividing $F_i$,
\item $F_i'=\{1,\dots,n\}$ if
$T^{F_i}$ is not in the image of the injection, and
$F_i'=\{1,\dots,n\}\setminus\{p\}$ if $T^{F_i}$ is the image of
$T^{F_i\cup\{p\}}$,
\item $K[F_i']$ is the polynomial ring in the indeterminates $X_q$,
$q\in F_i'$.
\end{itemize}
\end{proof}

One can try to convert the Hilbert decomposition indicated in the
proof of Theorem~\ref{Hupper} into a Stanley decomposition by the
following method. To simplify notation, we denote the element
$[i_1,\dots,i_k]$ by $w_G$ where $G=\{i_1,\dots,i_k\}$. We call
these elements \emph{generators} and the products $\mu w_G$, $\mu$ a
monomial in $R$, \emph{monomials}. In the multigraded structure of
the Koszul complex, the degree of $\mu w_G$ is $\mu X^G$ (where we
again identify a monomial with its exponent vector).

For each pair $(K[F_i'],F_i)$ in the decomposition, we now
choose a monomial $h_i=\mu w_G$ such that $\mu X^G=X^{F_i}$.
Let us call $h_i$ the \emph{hook} of $(K[F'_i],X^{F_i})$. In
the total set of monomials that we obtain by multiplying $h_i$
by the monomials in $K[F_i']$ and collecting over all $i$, each
multidegree appears with the right multiplicity (because we are
starting from a Hilbert decomposition). The crucial point is to
make these monomials (of the same degree) linearly independent
over $K$.

Note that each hook produces a given multidegree at most once. Fix a
multidegree, and consider all hooks that contribute to it. Each of
them has the form $\mu w_G$, and it is enough to make the family of
generators $w_G$ associated with the given multidegree linearly
independent over $R$ (Proposition~\ref{Hilb->St}).

For a given monomial $\mu$ in $R$, let the squarefree part
$\sqf(\mu)$ be the product of the indeterminates dividing $\mu$.
Clearly, a generator is associated with a given multidegree $\nu$ if
and only if it is associated with $\sqf(\nu)$ (since all hooks are
squarefree). This observation reduces the test for linear
independence to the squarefree degrees.

To prove the desired linear independence, we use the following simple
criterion: if we can order a family $(w_G)_{G\in\cG}$ in such a way
that $G_1\cup\dots\cup G_m\supsetneq G_1\cup\dots\cup G_{m-1}$ for
all $m$, then the family $\cG$ is linearly independent.

Let us now consider the special case $n=5$, $k=2$. One has
$\Stdepth_n M(5,2)\ge 3$ by Corollary~\ref{StdepthMk}, but in fact
$\Stdepth_n M(5,2)=4$, as we will see now. Following
\cite[p.~35]{StWhAA}, we obtain an injection $Y_3\to Y_2$ if we go
through the monomials $\mu$ in $Y_3$ lexicographically and choose
for each $\mu$ the lexicographically smallest divisor that is still
available: $123 \mapsto 12$, $124\mapsto 14$,\dots, $345\mapsto 34$.
Furthermore $12345\mapsto 1234$.

For the squarefree monomials of degree $2$, there is only a single
choice of hooks, namely the corresponding generator, and this leads
to no problem in degree $3$: if the total degree of a squarefree
monomial is $3$, then there are exactly two generators associated
with it, and they are automatically linearly independent.

Now we come to total degree $4$, and the choice of hooks
becomes critical. Consider $1234$. There are exactly $6$
monomials of this multidegree. Of these two are already in use,
namely $13[24]=X_1X_3[24]$ ($24$ is the image of $245$ in our
injection) and $12[34]$ ($345\mapsto 34$). Since $14[23]$ is
linearly dependent on the first two over $K$, it is also
excluded, and we choose $34[12]$ as the hook of $1234$. It is
``good,'' since $[24],[34],[12]$ are linearly independent over
$R$.

Further choices: $1235\mapsto 23[15]$, $1245\mapsto 25[14]$,
$1345\mapsto 45[13]$, $2345\mapsto 45[23]$. Again we get
linearly independent families of generators for each squarefree
multidegree of total degree $4$.

The generators associated with multidegree $12345$ are
$[15],[14],[13],[23]$. They are linearly independent, and we
are done.

Using Lemma~\ref{KoszInd}, one obtains that
\begin{equation}
\Stdepth M(n,n-3)=n-1
\end{equation}
for $n\ge 5$. We believe that $\Stdepth M(n,k)=n-1$ for all $k\ge
\lfloor n/2\rfloor$. It suffices to show this for $n$ odd,
$k=(n-1)/2$. The general statement would follow by induction.

In the lower half of the Koszul complex the situation is much more
complicated, and it seems impossible to give a precise, simple
expression even for $\Hdepth_1$. The proposition below provides a
trivial upper bound.

\begin{proposition}\label{Hbound}
Let $k<\lfloor n/2\rfloor$. Then
$$
\Hdepth_1 M(n,k) \le n-\left \lceil \frac{n-k}{k+1}\right\rceil.
$$
\end{proposition}

\begin{proof}
Simply consider the quotient of the second, negative term in the
numerator polynomial by the first term.
\end{proof}

Naively one might think that the proposition gives the correct value
as it does in the case $k=1$ (and for $k\ge\lfloor n/2\rfloor$). A
computer experiment confirms this value for $n\le 22$. However for
$n=23$ it fails for $k=3,4,5$. As we shall see in the next section,
the upper bound in Proposition~\ref{Hbound} is very far from the
truth, see Theorems~\ref{thm:regimeA} and \ref{thm:regimeB}. As a
preparatory step, we prove the following result, which, in
combination with \eqref{HD3}, forms the key for proving these
theorems.

\begin{proposition}\label{summation}
Let $Q_{n,k}$ be the numerator polynomial of the $\ZZ$-graded Hilbert
series of $M(n,k)$. Then
\begin{align} \label{eq:sum1}
\frac{Q_{n,k}}{(1-T)^s}&=\sum_{j=0}^\infty \left( (-1)^j\binom{n-s}{k+j}
+ \sum_{t=1}^s \binom{n-t}{k-1}\binom{s-t+j}{s-t}\right)T^{j+k}\\
\notag
&\kern-10pt
=\sum_{j=0}^\infty \Bigg( (-1)^j\binom{n-s}{k+j}\\
&\kern3cm
+ \sum_{\ell=0}^{k-1} \binom {j+\ell}\ell
\binom{n-s-j-\ell-1}{k-\ell-1}\binom{s+j+\ell}{s-1}\Bigg)T^{j+k}.
\label{eq:sum2}
\end{align}
\end{proposition}

\begin{proof}
By \eqref{NumSt}, equation~\eqref{eq:sum1} is true for $s=0$. For
the induction, one observes that the term in the inner sum is the
degree $j$ value of the Hilbert function of the free module of rank
$\binom{n-t}{k-1}$ over the polynomial ring in $s-t+1$ variables. In
other words, its sum over $j$ is the Hilbert series of this module.
Multiplication by $1/(1-T)$ increases the number of variables by
$1$. Thus the multiplication by $1/(1-T)$ replaces $s$ by $s+1$ in
these terms, as desired.

In order to complete the proof of \eqref{eq:sum1},
it remains to show that
$$
\frac1{1-T}\sum_{j=0}^\infty (-1)^j\binom{n-s}{k+j}T^{j+k}=
\sum_{j=0}^\infty
\left((-1)^j\binom{n-(s+1)}{k+j}+\binom{n-(s+1)}{k-1}\right)T^{j+k}.
$$
After the replacement of $n-s$ by $n$ this is the case $s=1$, and
the easy verification is left to the reader.

In order to establish the second form \eqref{eq:sum2}, we rewrite the
inner sum in \eqref{eq:sum1} as follows:
\begin{align*}
\sum_{t=1}^s \binom{n-t}{k-1}&\binom{s-t+j}{s-t}
=\sum_{t=1}^s (-1)^{k-1}\binom{-n+t+k-2}{k-1}\binom{s-t+j}{s-t}\\
&=\sum_{t=1}^s (-1)^{k-1}
\sum _{\ell=0} ^{k-1}\binom{-s-j+t-1}{\ell}\binom{s+j+k-n-1}{k-\ell-1}
\binom{s-t+j}{s-t}\\
&=\sum _{\ell=0} ^{k-1}\binom{n-s-j-\ell-1}{k-\ell-1}
\sum_{t=1}^s \binom{s+j+\ell-t}{\ell}\binom{s-t+j}{j}\\
&=\sum _{\ell=0} ^{k-1}\binom{n-s-j-\ell-1}{k-\ell-1}
\binom {j+\ell} \ell
\sum_{t=1}^s \binom{s+j+\ell-t}{j+\ell}\\
&=\sum _{\ell=0} ^{k-1}\binom{n-s-j-\ell-1}{k-\ell-1}
\binom {j+\ell} \ell
\binom{s+j+\ell}{j+\ell+1}.
\end{align*}
Here, to arrive at the second line and at the last line, we used special
instances of the Chu--Vandermonde summation
(cf.\ e.g.\ \cite[Sec.~5.1, (5.27)]{GrKPAA}).
\end{proof}

\begin{remark} \label{rem:imposs}
In hypergeometric terms (cf.\ \cite{SlatAC} for definitions), the
inner sums (over $t$ and $\ell$, respectively) in \eqref{eq:sum1} and
\eqref{eq:sum2} are $_3F_2$-series, namely
\begin{multline*}
\binom{n-1}{k-1}\binom{s+j-1}{s-1}
{} _{3} F _{2} \!\left [ \begin{matrix} { 1 - s, k - n, 1}\\ { 1 - j - s, 1 -
      n}\end{matrix} ; {\displaystyle 1}\right ]
\\=
\binom{n-s-j-1}{k-1}\binom{s+j}{s-1}
{} _{3} F _{2} \!\left [ \begin{matrix} { 1 - k, 1 + j, 1 + j + s}\\ { 2 + j, 1 + j
      - n + s}\end{matrix} ; {\displaystyle 1}\right ].
\end{multline*}
There are no summation formulas available for these $_3F_2$-series,
and therefore one cannot expect that they can be summed in closed
form. Indeed, by looking at special values of $k$ and $s$,
respectively by applying the Gosper--Zeilberger algorithm in order to
find a recurrence for these series and subsequently applying the
Petkov\v sek algorithm to the recurrence (cf.\ \cite{PeWZAA}),
one can {\it prove} that
these series cannot be further simplified. It is for this reason,
that, given $k$ and $n$, it is difficult to find the smallest $s$
such that {\it all\/} the coefficients in the polynomial
\eqref{eq:sum1} (respectively in \eqref{eq:sum2})
are non-negative, that
is, to find the Hilbert depth of $M(n,k)$ for the standard grading (cf.\
\eqref{HD3}).
\end{remark}

If we combine Proposition~\ref{summation} with \eqref{HD3}, then we
obtain a monotonicity property for the Hilbert depth
of the syzygy modules $M(n,k)$.

\begin{corollary}\label{monotone}
For all $k$ one has
$$
\Hdepth_1 M(n,k)\le \Hdepth_1 M(n,k+1).
$$
\end{corollary}

\begin{proof}
For $s$ fixed, the quotient of the negative terms on the right-hand
side of \eqref{eq:sum1} is smaller than the quotients of the
corresponding positive terms.
\end{proof}

\section{An asymptotic discussion}\label{Asympt}

\def\al{\alpha}
\def\be{\beta}
\def\de{\delta}
\def\ep{\varepsilon}
\def\ga{\gamma}
\def\Ga{\Gamma}
\def\({\left(}
\def\){\right)}

In view of the apparent impossibility (addressed in
Remark~\ref{rem:imposs}) of finding a compact expression for
$\Hdepth_1 M(n,k)$, the next best result that one can hope for is
asymptotic approximations of $\Hdepth_1 M(n,k)$ as $n$ becomes
large. This will be the subject of this final section.
Our results, given in Theorems~\ref{thm:regimeA} and \ref{thm:regimeB}
below, show that the general bounds in Corollary~\ref{StdepthMk} and
Proposition~\ref{Hbound} are far from
the truth for large $n$, that is, they can be substantially improved.
We shall discuss two ``regimes" for large $n$.
In the first part of this section,
we let $k$ be fixed, while $n$ tends to $\infty$.
On the other hand, in the second part, we let both $k$ and $n$ tend to
$\infty$ at a fixed rate.

\subsection{The case of fixed $k$ and large $n$}
The theorem below provides rather precise asymptotics for $\Hdepth_1
M(n,k)$ for the case where $k$ is fixed and $n$ tends to $\infty$.

\begin{theorem} \label{thm:regimeA}
For a fixed positive integer $k$, we have
\begin{multline} \label{eq:regimeA}
\Hdepth _1 M(n,k)=\frac {1} {2}n+\frac {1} {2}\sqrt{(k-1)n\log n}\\
+\frac {1} {4}\sqrt{\frac {(k-1)n} {\log n}}\log\log n
+o\left(\sqrt{\frac {n} {\log n}}\log\log n\right),
\end{multline}
as $n\to\infty$.
\end{theorem}

\begin{proof}
By Theorem~\ref{Biro}, we know that \eqref{eq:regimeA} is correct if
$k=1$. We may therefore assume that $k\ge2$ in the sequel.

In all of this proof, we let $k$ be fixed and
\begin{equation} \label{eq:sasy}
s=\frac {1} {2}n-\frac {1} {2}\sqrt{(k-1)n\log n}
-\de\sqrt{\frac {(k-1)n} {\log n}}\log\log n,\quad \quad
\text {as $n\to\infty$},
\end{equation}
where $\de$ is a fixed positive real number.
We shall prove that the quotient of
\begin{equation} \label{eq:sum2A}
\sum_{\ell=0}^{k-1} \binom {j+\ell}\ell
\binom{n-s-j-\ell-1}{k-\ell-1}\binom{s+j+\ell}{s-1}
\end{equation}
and
\begin{equation} \label{eq:binomA}
\binom{n-s}{k+j}
\end{equation}
is (asymptotically) less than $1$ {\it for some} $j$ if $\delta>1/4$,
and larger than $1$ {\it for all\/} $j$ with $0\le j\le
n-s-k$ if $\delta<1/4$. Clearly, in view of Proposition~\ref{summation} and
\eqref{HD3}, this would establish the assertion of the theorem.

In order to establish this claim, we proceed in several steps.
In the first step, we show that, for large $n$,
the summands in the sum \eqref{eq:sum2A} can be bounded by a constant
times the term for $\ell=k-1$,
so that it suffices to prove the above claim for the quotient
\begin{multline} \label{eq:quotA}
\binom {j+k-1} {k-1}\binom{s+j+k-1}{s-1}
\bigg/
\binom{n-s}{k+j}\\=
\frac {(j+1)_{k-1}} {(k-1)!}\,
\frac {\Ga(s+j+k)} {\Ga(s)}\,\frac {\Ga(n-s-j-k+1)} {\Ga(n-s+1)}.
\end{multline}
Here, $(j+1)_{k-1}$ is the standard notation for shifted factorials
(Pochhammer symbols),
$$(j+1)_{k-1}=(j+1)(j+2)\cdots (j+k-1),$$
and $\Ga(x)$ denotes the classical gamma function (cf.\ \cite{SlatAC}).

In the second step, we consider the right-hand side of
\eqref{eq:quotA} as a continuous function in the {\it real\/}
variable $j$, $0\le j\le n-s-k$, and we determine the (asymptotic)
value of $j$ for which the expression \eqref{eq:quotA} is minimal.
Finally, in the third step, we estimate \eqref{eq:quotA} as
$n\to\infty$ for this value of $j$. The conclusion will be that it
will be less than $1$ if $\de>1/4$, while it will be larger than $1$
if $\de<1/4$.

\smallskip
{\sc Step 1}. The quotient of the $(\ell+1)$-st and the $\ell$-th
summand in \eqref{eq:sum2A} equals
\begin{equation} \label{eq:l+1/l}
\frac {(j+\ell+1)} {(\ell+1)}\,
\frac {(k-\ell-1)} {(n-s-j-\ell-1)}\,
\frac {(s+j+\ell+1)} {(j+\ell+2)},\quad \quad
\ell=0,1,\dots,k-2,
\end{equation}
for which we have
\begin{align*}
\frac {(j+\ell+1)} {(\ell+1)}\,
\frac {(k-\ell-1)} {(n-s-j-\ell-1)}\,&
\frac {(s+j+\ell+1)} {(j+\ell+2)}\\
&\ge
\frac {j+1} {(k-1)}\,
\frac {1} {(n-s-1)}\,
\frac {(s+1)} {(j+k)}\\
&\ge \frac {1} {2k(k-1)}
\end{align*}
for $n$ large enough, where we have taken into account our choice
\eqref{eq:sasy} of $s$. Hence, all the summands in \eqref{eq:sum2A}
are bounded by a constant times the term for $\ell=k-1$.

\smallskip
{\sc Step 2}. The reader should recall that $k$ is fixed, $s$ is
given by \eqref{eq:sasy}, and that we consider large $n$.
It is a simple fact that the product
$$\Ga(s+j+k)\,\Ga(n-s-j-k+1)$$
occurring in \eqref{eq:quotA} attains its minimum when the arguments
of the gamma functions are equal to each other, that is, for
$j=\frac {n+1} {2}-s-k$. It is then not difficult to see that this
implies that, as a function in $j$, the expression \eqref{eq:quotA}
cannot attain its minimum at the boundary of the defining interval
for $j$, that is, at $j=0$ or at $j=n-s-k$. (The term $(j+1)_{k-1}$
cannot compensate the difference in orders of magnitude of
\eqref{eq:quotA} at $j=\frac {n+1} {2}-s-k$ and at $j=0$,
respectively at $j=n-s-k$.) Therefore, in order to determine places
of minima of the function \eqref{eq:quotA} (in $j$), we compute its
logarithmic derivative with respect to $j$, which we shall
subsequently equate to $0$. Let $\psi(x)$ denote the classical {\it
digamma function}, which, by definition, is the logarithmic
derivative of the gamma function. Using this notation, the
logarithmic derivative of \eqref{eq:quotA} is given by
\begin{equation*} 
\sum _{i=1} ^{k-1}\frac {1} {j+i}
+\psi(s+j+k)-\psi(n-s-j-k+1).
\end{equation*}
Let us for the moment write $s=\frac {n} {2}-s_1$, where $s_1=o(n)$,
for short. Then, equating the above logarithmic derivative to $0$
means to solve the equation
\begin{equation} \label{eq:quotder}
\sum _{i=1} ^{k-1}\frac {1} {j+i}
+\psi\(\frac {n} {2}-s_1+j+k\)-\psi\(\frac {n} {2}+s_1-j-k+1\)=0
\end{equation}
for $j$. For our purposes, it will not be necessary to determine
solutions $j$ exactly (which is impossible anyway), but it suffices
to get appropriate asymptotic estimates.

For the following considerations we need the first few terms in
the asymptotic series for
the digamma function (cf.\ \cite[1.18(7)]{ErdeAA}):
\begin{equation} \label{eq:psi}
\psi(x)=\log x-\frac {1} {2x}+O\(\frac {1}
{x^2}\),
\quad \quad \text {as $x\to\infty$}.
\end{equation}
If we suppose that $j\sim\al n$ as $n\to\infty$, where $\al>0$,
then, using \eqref{eq:psi}, the limit of the left-hand side of
\eqref{eq:quotder} as $n\to\infty$ can be computed: it equals
$$\log\frac {\frac {1} {2}+\al} {\frac {1} {2}-\al}\ne0,$$
a contradiction to the equation \eqref{eq:quotder}. Hence, we must
have $j=o(n)$ as $n\to\infty$.

Let us, for convenience, write $j=s_1+j_1$. Then \eqref{eq:quotder}
becomes
\begin{equation} \label{eq:quotder2}
\sum _{i=1} ^{k-1}\frac {1} {s_1+j_1+i}
+\psi\(\frac {n} {2}+j_1+k\)-\psi\(\frac {n} {2}-j_1-k+1\)=0.
\end{equation}
Using \eqref{eq:psi}, the estimate
\begin{align*}
\log \(\frac {n} {2}+j_1+k\)&=\log \frac {n} {2}
+\log\(1+\frac {2(j_1+k)} {n}\)\\
&=\log \frac {n} {2}
+\frac {2(j_1+k)} {n}+O\(\frac {j_1^2} {n^2}\),
\quad \quad \text{as $n\to0$},
\end{align*}
and an analogous estimate for $\log\big(\frac {n} {2}-j_1-k+1\big)$,
the left-hand side of \eqref{eq:quotder2} is
asymptotically
\begin{equation} \label{eq:quotderasy}
\frac {k-1} {s_1+j_1}+O\(\frac {1} {(s_1+j_1)^2}\)
+\frac {4j_1} {n}+O\(\frac {j_1^2} {n^2}\).
\end{equation}
If the equation \eqref{eq:quotder2} wants to be true, then the
asymptotically largest terms in \eqref{eq:quotderasy} must cancel
each other. If we suppose that $j_1\ll \sqrt{n/\log n}$, then, taking
into account that $s_1\sim \frac {1} {2}\sqrt{(k-1)n\log n}$, the
term $(k-1)/(s_1+j_1)$ would be asymptotically strictly larger than all
other terms in \eqref{eq:quotderasy}, a contradiction. On the other
hand, if we suppose that $j_1\gg \sqrt{n/\log n}$, then the
term $4j_1/n$ would be asymptotically strictly larger than all
other terms in \eqref{eq:quotderasy}, again a contradiction. Hence, we must
have $j_1\sim \al \sqrt{n/\log n}$ for some $\al>0$. If we
substitute this in \eqref{eq:quotderasy} and equate (asymptotically)
the first and the third term in this expression, then we obtain
$\al=-\frac {1} {2}\sqrt{k-1}$.

In summary, under our assumptions, the value(s)\footnote{It could be
proved, using estimates for the derivative of \eqref{eq:quotder} with
respect to $j$, that, for $n$ large enough, there is a unique zero of
\eqref{eq:quotder}, and, hence, a unique $j$ which minimizes
\eqref{eq:quotA}. Since we do not really need this fact, we omit its
proof.} for $j$ which
minimize the expression \eqref{eq:quotA} is (are) asymptotically equal to
\begin{equation} \label{eq:j0}
j_0=\frac {1} {2}\sqrt{(k-1)n\log n}\(1-\frac {1} {\log n}+o\(\frac
{1} {\log n}\)\), \quad \quad \text {as $n\to\infty$}.
\end{equation}

\smallskip
{\sc Step 3}. Now we substitute \eqref{eq:sasy} and \eqref{eq:j0} in
\eqref{eq:quotA}, and determine the asymptotic behaviour of the
resulting expression. For the term $(j_0+1)_{k-1}$, we use the
estimation
\begin{align*}
\log \big((j_0+1)&_{k+1}\big)=\log\(j_0^{k-1}\(1+O\(\frac {1} {j_0}\)\)\)\\
&=(k-1)\log j_0+O\(\frac {1} {j_0}\)\\
&=(k-1)\(\frac {1} {2}\log n+\frac {1} {2}\log\log n
+\frac {1} {2}\log(k-1)-\log2\)+o\(1\).
\end{align*}
In order to approximate the gamma functions in \eqref{eq:quotA},
we need Stirling's formula in the
form
\begin{equation} \label{eq:Stirling}
\log \Ga(x)=\(x-\frac {1} {2}\)\log x-x+\frac {1} {2}\log(2\pi)+o\(1\),
\quad \quad \text{as $x\to\infty$}.
\end{equation}
Writing, as earlier,
$s=\frac {n} {2}-s_1$, application of \eqref{eq:Stirling} to the term
$\Ga(s)$ gives
\begin{align*}
\Ga(s)&=\(\frac {n} {2}-s_1-\frac {1} {2}\)\log\(\frac {n} {2}-s_1\)
-\(\frac {n} {2}-s_1\)+\frac {1} {2}\log(2\pi)+o\(1\)\\
&=\(\frac {n} {2}-s_1-\frac {1} {2}\)\(\log\frac {n} {2}
-\frac {2s_1} {n}-\frac {1} {2}\frac {(2s_1)^2} {n^2}+
O\(\frac {s_1^3} {n^3}\)\)\\
&\kern6cm
-\(\frac {n} {2}-s_1\)+\frac {1} {2}\log(2\pi)+o\(1\).
\end{align*}
The terms $\Ga(s+j+k)$, $\Ga(n-s-j-k+1)$, and $\Ga(n-s+1)$ are treated
analogously. If everything is put together, after a considerable
amount of simplification, we obtain
\begin{multline} \label{eq:quotasy}
\frac {(j+1)_{k-1}} {(k-1)!}\,
\frac {(s+j+k-1)!} {(s-1)!}\,\frac {(n-s-j-k)!} {(n-s)!}\\
=\exp\bigg((k-1)\(\frac {1} {2}-2\de\)\log\log n
+\frac {1} {2}(k-1)\log(k-1)\kern3cm\\
-(k-1)\log 2-\log\big((k-1)!\big)
+o(1)\bigg),\quad \quad \text {as $n\to\infty$}.
\end{multline}
We can now clearly see that the right-hand side is $\ll 1$
if $\delta>1/4$, and $\gg 1$ if $\delta<1/4$. This completes the
proof of the theorem.
\end{proof}

\begin{remark}
The alert reader may wonder whether the estimation given
by the right-hand side of \eqref{eq:regimeA} provides a lower or
upper bound for $\Hdepth _1 M(n,k)$. We shall now show that
(at least for $k\ge4$) it is indeed an upper bound, that is, for
fixed $k\ge4$ and large enough $n$, we have
\begin{equation} \label{eq:regimeA1/4}
\Hdepth _1 M(n,k)\le\frac {1} {2}n+\frac {1} {2}\sqrt{(k-1)n\log n}
+\frac {1} {4}\sqrt{\frac {(k-1)n} {\log n}}\log\log n.
\end{equation}

To see this, we return to
\eqref{eq:l+1/l}, which expresses the quotient of the $(\ell+1)$-st and
the $\ell$-th summand in \eqref{eq:sum2A}. Using this expression,
we see that, for $j=j_0$ (cf.\ \eqref{eq:j0}),
this quotient is asymptotically equal to
$$
\frac {k-\ell-1} {\ell+1}+o(1),\quad \quad
\ell=0,1,\dots,k-2,
$$
as $n\to\infty$. If we denote the $\ell$-th summand in the sum
\eqref{eq:sum2A} by $t_\ell$, then this implies that
$$t_\ell=t_{k-1}\(\binom {k-1}\ell + o(1)\).$$
Thus, we infer that the sum in \eqref{eq:sum2A} is asymptotically
equal to
$$
\sum _{\ell=0} ^{k-1}t_\ell=t_{k-1}
\sum _{\ell=0} ^{k-1}\(\binom {k-1}\ell + o(1)\)
=\binom {j+k-1}{k-1}\binom{s+j+\ell}{s-1}\big(2^{k-1}+o(1)\big).
$$
If we combine this with \eqref{eq:quotasy}, in which we computed the
asymptotics of the quotient of $t_{k-1}$ and \eqref{eq:binomA} with
$s$ given by \eqref{eq:sasy} and $j=j_0$, then
we obtain that the quotient of the sum \eqref{eq:sum2A}
and the binomial coefficient \eqref{eq:binomA}, where $s$ is given by
\eqref{eq:sasy} with $\de$ specialized to $1/4$ and $j=j_0$, is equal to
$$\exp\bigg(\frac {1} {2}(k-1)\log(k-1)-\log\big((k-1)!\big)
+o(1)\bigg),\quad \quad \text {as $n\to\infty$}.
$$
It is not difficult to show that, for $k\ge4$, we have
$$\frac {1} {2}(k-1)\log(k-1)-\log\big((k-1)!\big)<0.$$
In view of \eqref{HD3}, this implies \eqref{eq:regimeA1/4}.
We expect the same to be true as well for $k=2$ and $k=3$, but we did
not perform the necessary asymptotic calculations using
longer asymptotic series.
\end{remark}

\subsection{The case of large $n$ and $k$}
In this part, we consider the case where both $k$ and $n$ tend to
$\infty$ at a fixed rate, say $k= \be n+o(n)$ with $ \be >0$. We
shall see that then $\Hdepth_1 M(n,k)\sim (1-\ga) n$, where $\ga\le
\frac {1} {2}-\be$. (See Theorem~\ref{thm:regimeB} for the exact
definition of $\ga$, and Remarks~\ref{rem:regimeB}.(2) for a graph
of $\ga$ as a function in $\be$.) Note that this estimate is an
(asymptotic) improvement of Corollary~\ref{StdepthMk}, which only
yields $\Hdepth_1 M(n,k)\ge \(\frac {1} {2}+\frac{ \be
}{2}+o(1)\)n$.

Again, our starting point is \eqref{HD3} in combination with
Proposition~\ref{summation}. We begin by providing an asymptotic estimate
for the isolated binomial coefficient in \eqref{eq:sum1}.

\begin{lemma} \label{lem:3}
Let $k= \be n+o(n)$, $j=\al n+o(n)$, and $s= \ga  n+o(n)$,
where $\al, \be , \ga $ are positive real numbers not exceeding $1$.
Then, as $n\to\infty$, we have
\begin{multline}
\binom {n-s} {k+j}=\(\frac {(1- \ga )^{1- \ga }}
{(\al+ \be )^{\al+ \be }
(1-\al- \be - \ga )^{1-\al- \be - \ga }}\)^{n}\\
\times\text {\em asymptotically smaller terms}.
\label{eq:asy7}
\end{multline}
\end{lemma}

\begin{proof}
This is a simple consequence of Stirling's formula
\eqref{eq:Stirling}.
\end{proof}

Next, we provide an asymptotic estimate
for the inner sum in \eqref{eq:sum1}.

\begin{lemma} \label{lem:4}
Let $k= \be n+o(n)$, $j=\al n+o(n)$, and $s= \ga  n+o(n)$,
where $\al, \be , \ga $ are positive real numbers not exceeding $1$.
Then, as $n\to\infty$, we have
\begin{multline}
\sum _{t=1} ^s\binom {n-t} {k-1}\binom {s+j-t} {s-t}=
\(
\frac {(\al+ \ga )^{\al+ \ga }}
{\al^{\al} \be ^{ \be } \ga ^{ \ga }(1- \be )^{1- \be }}\)
^{n}\\
\times\text {\em asymptotically smaller terms}.
\label{eq:asy8}
\end{multline}
\end{lemma}

\begin{proof}
The summands of the sum on the left-hand side of \eqref{eq:asy8} are
monotone decreasing in $t$. In particular, they are bounded above by the
summand with $t=1$. Stirling's formula \eqref{eq:Stirling} applied to
this summand yields the approximation given on the right-hand side of
\eqref{eq:asy8}, and since the number of summands is $s\sim  \ga n$,
the approximation given there is also valid for the whole sum.
\end{proof}

\begin{theorem} \label{thm:regimeB}
Let $\be$ be a positive real number with $\be\le1/2$.
For $k= \be n+o(n)$, we have
\begin{equation} \label{eq:regimeB}
\Hdepth _1 M(n,k)=\(1-\ga\) n+o(n),\quad \quad \text {as $n\to\infty$},
\end{equation}
where $\ga$ is the smallest nonnegative solution of the equation
\begin{equation} \label{eq:Gl}
\frac {(\al+ \ga )^{\al+ \ga }
(\al+ \be )^{\al+ \be }
(1-\al- \be - \ga )^{1-\al- \be - \ga }}
{\al^{\al} \be ^{ \be } \ga ^{ \ga }(1- \be )^{1- \be }
(1- \ga )^{1- \ga }}=1,
\end{equation}
with
$$\al=\frac {1}
{4}\big(1-2\be-2\ga+\sqrt{(1-2\be-2\ga)^2-8\be\ga}\big).$$
\end{theorem}

\begin{remarks} \label{rem:regimeB}
(1) Our computer calculations show that there is always {\it exactly
one} solution to \eqref{eq:Gl}. More precisely, as a function in
$\ga$, the left-hand side of \eqref{eq:Gl} seems always to be a
monotone increasing function. In view of the daunting expression that
one obtains by substituting the indicated value of $\al$ in
\eqref{eq:Gl}, we did not try to prove this observation since this is
also not essential for the assertion of Theorem~\ref{thm:regimeB}.

(2) It is not difficult to see that the value $\ga$ in
Theorem~\ref{thm:regimeB} satisfies $\ga\le 1/2-\be$. Indeed,
except for $\be$ close to $0$ or close to $1/2$, this is
so by a large margin, as the graph in Figure~\ref{fig1} shows.
As we already remarked at the beginning of this part, this yields a
considerable improvement over the bound implied by
Corollary~\ref{StdepthMk}.

\begin{figure}[htb]

\vbox{
\leavevmode
\hbox{\hskip5cm}

\vskip4.5cm
\leavevmode

\centerline{\hskip-4.8cm
\includegraphics[width=0.25\linewidth,height=0.25\linewidth]{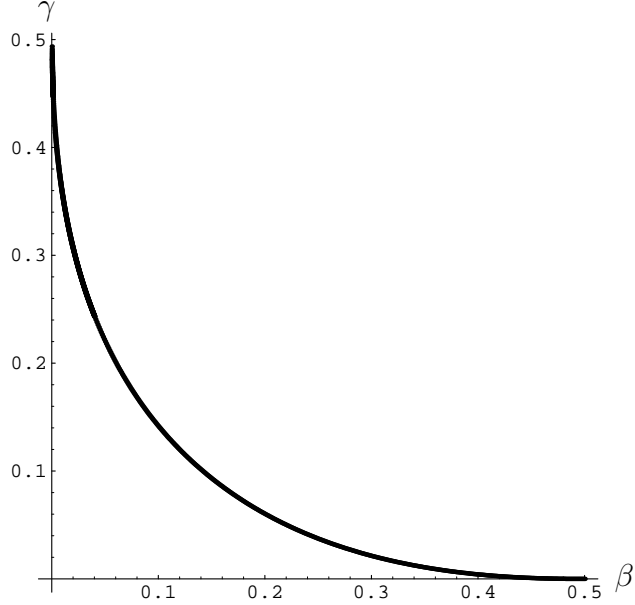}
}
\vskip-8.4cm
\centerline{\hskip3.5cm$\ga$\hfill}
\vskip7.1cm
\centerline{\hskip7.6cm$\be$}
\vskip10pt
}

\caption{The value $\ga$ as a function in $\be$}
\label{fig1}
\end{figure}

\end{remarks}

\begin{proof}[Proof of Theorem~\ref{thm:regimeB}]
We proceed in a manner similar to the proof of
Theorem~\ref{thm:regimeA}. First, we form the quotient of
\eqref{eq:asy8} and \eqref{eq:asy7}, which is asymptotically
\begin{multline} \label{eq:quotB}
\(\frac {(\al+ \ga )^{\al+ \ga }
(\al+ \be )^{\al+ \be }
(1-\al- \be - \ga )^{1-\al- \be - \ga }}
{\al^{\al} \be ^{ \be } \ga ^{ \ga }(1- \be )^{1- \be }
(1- \ga )^{1- \ga }}\)^n\\
\times\text {asymptotically smaller terms},
\end{multline}
as $n\to\infty$. In view of \eqref{HD3}, we need to find the smallest
$\ga$ such that the base of the exponential in \eqref{eq:quotB} is
larger than $1$ for all $\al$ with $0\le \al\le 1-\be-\ga$.
Hence, we should next consider this base,
$$
f(\al):=\frac {(\al+ \ga )^{\al+ \ga }
(\al+ \be )^{\al+ \be }
(1-\al- \be - \ga )^{1-\al- \be - \ga }}
{\al^{\al} \be ^{ \be } \ga ^{ \ga }(1- \be )^{1- \be }
(1- \ga )^{1- \ga }},
$$
and, given $\be$ and $\ga$, discuss it as a function in $\al$.
More precisely, our goal is to determine the value(s) of $\al$ for
which $f(\al)$ attains its minimum. In a subsequent step, we shall
have to find the smallest possible $\ga$ such that this minimum is
at least $1$.

Our first observation is that both
$$f(0)=\frac {(1- \be - \ga )^{1- \be - \ga }}
{ (1- \be )^{1- \be }
(1- \ga )^{1- \ga }}$$
and
$$
f(1-\be-\ga)=\frac 1
{(1-\be-\ga)^{1-\be-\ga} \be ^{ \be } \ga ^{ \ga }}
$$
are at least $1$. For $f(1-\be-\ga)$ this is totally obvious since
$\be$, $\ga$, and $1-\be-\ga$ are numbers between $0$ and $1$, while
for $f(0)$ this follows from the facts that $f(0)\big\vert_{\ga=0}=1$
and that $f(0)$ is monotone increasing in $\ga$. Consequently, for
given $\be$ and $\ga$, the minimum of $f(\al)$ is either at least
$1$, or it is attained in the interior of the interval
$[0,1-\be-\ga]$. In order to find the places of minima in the
interior of this interval, we compute the logarithmic derivative
of $f(\al)$,
\begin{equation} \label{eq:fder}
\frac {d} {d\al}\log f(\al)=
\log\frac {(\al+ \ga )(\al+ \be )}
{\al(1-\al- \be - \ga )},
\end{equation}
and equate it to $0$. This equation leads to a quadratic equation in
$\al$ with solutions
\begin{equation} \label{eq:alpha}
\al=\frac {1}
{4}\big(1-2\be-2\ga\pm\sqrt{(1-2\be-2\ga)^2-8\be\ga}\big).
\end{equation}
Since, from \eqref{eq:fder}, we see that the derivative of $f(\al)$
is $+\infty$ at $\al=0$ and at $\al=1-\be-\ga$, the smaller of the
two solutions in \eqref{eq:alpha} must be the place of a local maximum of
$f(\al)$, while the larger solution must be the place of a local minimum
(if they are at all real numbers).

Finally, we must find the smallest $\ga$ such that the minimum of
$f(\al)$, for $\al$ ranging in the interval $[0,1-\be-\ga]$, is at
least $1$. In particular, the above described local minimum must be
at least $1$. Hence, we must substitute the larger value given by
\eqref{eq:alpha}, $\al_0$ say,
in \eqref{eq:quotB}, and restrict our search to
those values of $\ga$, where the result is at least $1$.

Now, if we do this substitution in \eqref{eq:quotB}
(the reader should observe that the result is exactly the left-hand
side of \eqref{eq:Gl}) and set $\ga=0$,
then we obtain
$$
f(\al_0)\big\vert_{\ga=0}=\frac {1} {2\,(1-\be)^{1-\be}}\le
\frac {1} {2\,(1/e)^{1/e}}=0.72\ldots<1.
$$
Hence, the smallest $\ga$ such that the minimum of
$f(\al)$, for $\al$ ranging in the interval $[0,1-\be-\ga]$, is at
least $1$ is indeed the solution to the equation \eqref{eq:Gl}.
\end{proof}

\end{document}